\documentclass[11pt]{article}

\usepackage{amsmath, amssymb, amsthm, mathrsfs}
\usepackage{geometry, graphicx}
\usepackage{color, hyperref}
\usepackage{tikz} 

\hypersetup{
    colorlinks=true,
    linkcolor=blue,
    citecolor=red,
    urlcolor=cyan
}

\newtheorem{theorem}{Theorem}[section]

\newtheorem{proposition}[theorem]{Proposition}

\theoremstyle{definition}
\newtheorem{definition}[theorem]{Definition}
\newtheorem{remark}[theorem]{Remark}

\newcommand{\R}{\mathbb{R}}
\newcommand{\C}{\mathbb{C}}
\newcommand{\norm}[1]{\left\| #1 \right\|}
\newcommand{\bmo}{\mathrm{bmo}}
\newcommand{\vmo}{\mathrm{vmo}}
\newcommand{\om}{\boldsymbol{\omega}}
\newcommand{\xiVec}{\boldsymbol{\xi}}
\newcommand{\uVec}{\boldsymbol{u}}

\title{\textbf{Logarithmic Depletion of Vortex Stretching and Singularity Evasion in the 3D Navier-Stokes Equations}}
\author{Zoran Gruji\'c \\  \\  UAB }

\date{\today}

\begin{document}

\maketitle

\begin{abstract}
We present a geometric-analytic mechanism for the suppression of finite-time singularities in the 3D incompressible (unforced) Navier-Stokes equations for critical point singularities exhibiting $L^{3/2, \infty}$ spatial concentration of vorticity. We demonstrate that if the vorticity direction resides locally in a logarithmically weighted space of bounded mean oscillations, $\mathrm{bmo}_{1/|\log r|}$ -- a space failing the Dini condition and thus permitting wild oscillatory defects  -- the nonlinear vortex stretching is fundamentally depleted. By isolating a unidirectional geometric cancellation, we recast the stretching eigenvalue as a singular integral commutator. Utilizing a localized Coifman-Rochberg-Weiss estimate coupled with dyadic BMO tail bounds, we prove the stretching potential vanishes as a logarithmic envelope on shrinking super-level sets. This depletion forces the vorticity magnitude into a sub-critical Lorentz-Zygmund space via interpolated De Giorgi energy method. The logarithmic gain is subsequently transferred to the velocity field, forcing the geometric scale of local 1D sparseness below the uniform radius of spatial analyticity, ultimately averting the finite time blow-up via the harmonic measure maximum principle.
\end{abstract}

\tableofcontents
\vspace{0.5cm}

\section{Introduction}

The question of whether the 3D Navier-Stokes equations (NSE) can exhibit spontaneous formation of a singularity (no external force) remains a
fundamental open problem in mathematical physics. 
 From the perspective of fluid mechanics, potential singularity scenarios are typically modeled on physical configurations designed to maximize vortex stretching, such as the Moffatt-Kimura scenario of colliding vortex rings \cite{Moffatt2019, Moffatt2019R, Moffatt2023, Yao2020, Yao2020A, GrujicPAFA}. In this paper, we abstract the geometric-analytic nature of this mechanism to study a broader class of critical point singularities. More precisely, we focus on the critical regime where the vorticity magnitude concentrates at a rate of $|\om(x)| \sim |x|^{-2}$, placing it in the critical Lorentz space $L^{3/2, \infty}(\R^3)$.
\medskip

Looking for Navier-Stokes singularities in the critical realm -- and in particular, via self-similar profiles -- has been a deeply active area of research since Leray's pioneering work in the 1930s \cite{Leray1934}. Self-similar profiles in $L^3(\mathbb{R}^3)$ were ruled out by Ne\v{c}as, R\r{u}\v{z}i\v{c}ka and \v{S}ver\'ak \cite{Necas1996}, while profiles satisfying the local energy inequality were eliminated by Tsai \cite{Tsai1998}. Consequently, a major thrust of modern research seeks to identify asymptotically self-similar (continuously or discretely) blow-up profiles, theoretically, computationally, and recently via hybrid PINN methods \cite{Hou2023, DeepMind202X}.

\medskip

On the other hand, one can drop self-similarity altogether and ask whether simple boundedness in a critical space could prevent possible singularity formation. Escauriaza, Seregin and \v{S}ver\'ak established that $L^\infty\bigl((0, T), L^3(\R^3)\bigr)$ -- the hard endpoint of the Ladyzhenskaya-Prodi-Serrin scale \cite{Prodi1959, Serrin1962, Ladyzhenskaya1967} -- is a regularity class for the velocity \cite{ESS2003}, shifting focus to the Lorentz endpoints, $L^{3, \infty}(\R^3)$ for the velocity and $L^{3/2, \infty}(\R^3)$ for the vorticity. In this context, Seregin \cite{Seregin2015} demonstrated regularity under an assumption of smallness; see also the work by Barker and Prange \cite{BarkerPrange2021}. Furthermore, several recent formalisms are pointing -- in different ways -- to a critical regime: a recent work by Constantin, Ignatova and Vicol \cite{CIV202X} on putative Euler singularities, recent work by Gruji\'c and Xu \cite{Grujic2019, Grujic2020} ruling out asymptotically self-similar blow-up for hyper-dissipative Navier-Stokes models, as well as the fluid mechanics bookkeeping indicating that in a typical cascade construction \cite{Dascaliuc2013} leading to a potential singularity, the viscous mechanics forces the construction into a critical envelope.

\medskip

The geometric approach to breaking this criticality relies on understanding the local spatial coherence of the vorticity direction $\xiVec(x) = \om(x)/|\om(x)|$. Constantin and Fefferman \cite{Constantin1994, Constantin1993} famously demonstrated (for general flows) that if the vorticity direction is sufficiently regular (Lipschitz) in regions of the high vorticity magnitude, blow-up is prevented. The Lipschitz condition was subsequently weakened to $\frac{1}{2}$-H\"older by Beir\~ao da Veiga and Berselli \cite{BdVBe02}. In the case of Type-I blow-up, Giga and Miura \cite{Giga2011} significantly lowered this threshold, establishing that uniform continuity of the direction field is sufficient to rule out singularity formation; see also \cite{BdVGiGr16}. However, the requirement of uniform continuity forbids the presence of topological defects or singular phase structures within the highly concentrated vortex core.

\medskip

In this paper, we establish a rigorous geometric-analytic mechanism that breaks the critical $L^{3/2, \infty}$ scaling while assuming a substantially weaker condition that permits structural discontinuity. We assume only that the vorticity direction resides locally in a weighted space of bounded mean oscillations, $\bmo_{\phi}$, with a single-logarithmic weight $\phi(r) = \frac{1}{|\log r|}$. Because the weight fails the Dini condition ($\int_0 \frac{\phi(r)}{r} dr = \infty$), this functional framework supports highly oscillatory phase defects (e.g., configurations of the form $\sin(\log|\log|x||)$) and bypasses the classical continuity requirements of previous geometric criteria.

\medskip

The primary contribution of this paper is demonstrating that this singular, purely geometric condition is sufficient to deplete the vortex stretching. By isolating an exact unidirectional cancellation in the Biot-Savart law, we recast the stretching eigenvalue as a singular integral commutator. Coupling this with localized Coifman-Rochberg-Weiss singular integral estimates and dyadic BMO tail bounds, we show that the localized restricted Lorentz norm of the stretching eigenvalue vanishes. This forces the nonlinear stretching into a sub-critical envelope. Utilizing Lorentz real interpolation via De Giorgi energy method, we obtain a logarithmic improvement of the decay of the vorticity distribution function. (This first part of the paper could be thought of as a conceptual cousin of \cite{Bradshaw2015} where it was shown that the assumption on the logarithmic decay of the local mean oscillations of the vorticity direction upgraded Constantin's \emph{a priori} $L^1$ bound on the vorticity \cite{Constantin1990} to $L \log L$.) Next, this logarithmically improved decay of the distribution function is transferred from the vorticity to the velocity by O'Neil's convolution lemma, driving the local scale of sparseness of the velocity super-level sets into the dissipation range delineated by a lower bound on the radius of spatial analyticity. The endgame is provided by the harmonic measure maximum principle.

\medskip

This paper constitutes the first part of a two-part geometric-analytic framework. While the present work establishes that the singular $\bmo_{1/|\log r|}$ condition is sufficient -- in the case of a critical point singularity -- to trigger depletion of vortex stretching and avert finite time blow-up, a fundamental physical question remains: are the mechanics of the Navier-Stokes equations capable of propagating this specific topological regularity? In our companion paper \cite{GrujicProj2}, we address this question by analyzing the evolution of the vorticity direction as a perturbed Harmonic Map Heat Flow (HMHF) into the sphere (HMHF plus fluid transport plus repulsive drift born out of cross-diffusion plus tangential strain).

\section{Preliminaries and Functional Framework}

\subsection{Critical Point Singularities}
In order to focus our analysis on the critical endpoint of the functional scaling, we define the geometric profile of interest.

\begin{definition}\label{def:critical_point}
In what follows, a \emph{critical point singularity} refers to a potential finite-time blow-up profile centered at a spatial point (without loss of generality, the origin), where the vorticity magnitude locally scales as $O(|x|^{-2})$ and naturally inhabits the critical Lorentz space $L^{3/2, \infty}(\R^3)$. We factor this profile as $\omega(x,t) = \Phi(x,t)|x|^{-2}$, where the shape factor $\Phi(x,t)$ captures spatial variations (e.g., purely angular profiles $\Phi \equiv U(x/|x|)$ or oscillatory phases like $\Phi(x)=2+\sin\bigl(\log(1/|x|)\bigr)$ capturing the physics of discrete self-similarity). We assume this shape factor is either scale-invariant or log-periodic at the core, ensuring that the profile itself is uniformly bounded in the space-time while its spatial gradient satisfies the critical bound $|\nabla \Phi| \lesssim |x|^{-1}$ uniformly in time. As a direct consequence, for large amplitude thresholds $\lambda$, the super-level set $A_\lambda(t) = \{x \in \R^3 : \omega(x,t) > \lambda\}$ is contained within a spatial ball $B_R$ of radius $R \le C \lambda^{-1/2}$, where the geometric bounding constant $C$ is independent of time.
\end{definition}

\subsection{Lorentz Spaces}
For a measurable function $f$ on $\R^3$, let $f^*(s)$ denote its decreasing rearrangement. The Lorentz space $L^{p,q}(\R^3)$ is defined by the finiteness of the norm:
\begin{equation}
\norm{f}_{L^{p,q}} = \left( \int_0^\infty \big[ s^{1/p} f^*(s) \big]^q \frac{ds}{s} \right)^{1/q}, \quad 1 \le p < \infty, 1 \le q < \infty,
\end{equation}
with the standard modification for $q = \infty$: $\norm{f}_{L^{p,\infty}} = \sup_{s > 0} s^{1/p} f^*(s)$. We frequently utilize the sharp H\"older inequality in Lorentz spaces \cite{Hunt1966}: $\norm{fg}_{L^1} \le \norm{f}_{L^{p, \infty}} \norm{g}_{L^{p', 1}}$ where $1/p + 1/p' = 1$.

\subsection{Log-Weighted Bounded Mean Oscillations}
We consider a local log-weighted space of bounded mean oscillations \cite{Goldberg1979, Bradshaw2015}. For a locally integrable function $f$ and a radius scale $0 < r < 1/2$, we define the generalized $\bmo_{\phi}$ norm as:
\begin{equation}
\norm{f}_{\bmo_\phi} = \norm{f}_{L^\infty} + \sup_{x \in \R^3, 0 < r < 1/2} \frac{1}{\phi(r)} \frac{1}{|B_r(x)|} \int_{B_r(x)} |f(y) - c_{B_r(x)}| dy,
\end{equation}
where $c_{B_r(x)}$ is the local mean over the ball. Because the vorticity direction field is a unit vector field, the base Lebesgue norm trivially evaluates to $1$, anchoring the field globally without risking volume measure divergence. For the primary exposition of this paper, we choose the single-logarithmic weight $\phi(r) = \frac{1}{|\log r|}$.

\begin{remark}[The Dini Divergence and Topological Singularities]\label{rem:dini}
A fundamental characterization by Spanne and Campanato \cite{Campanato1963} establishes that a function in generalized $\bmo_\phi$ is uniformly continuous if and only if the weight $\phi$ satisfies the Dini condition, $\int_0 \frac{\phi(r)}{r} dr < \infty$. For our chosen geometric weight:
\begin{equation}
\int_0^\delta \frac{\phi(r)}{r} dr = \int_0^\delta \frac{1}{r |\log r|} dr = \Big[ -\log |\log r| \Big]_0^\delta = \infty.
\end{equation}
Hence, $\xiVec \in \bmo_{1/|\log r|}$ is a fundamentally weaker geometric condition than uniform continuity. It permits structural defects and intense phase singularities (e.g., oscillating topological configurations of the form $\sin(\log|\log|x||)$), bypassing the classical continuity requirements of \cite{Giga2011}.
\end{remark}

\begin{remark}[Phenomenological Context: Viscous $\vmo$ vs. Inviscid $\bmo_1$]
From the perspective of fluid phenomenology and turbulence scaling laws, being permanently trapped in an unweighted $\bmo_1$ space (local space of bounded mean oscillations) is an  \emph{inviscid artifact}. Geometrically, a scale-invariant $\bmo_1$ singularity corresponds to the formation of a sharp topological defect---a dipole bubble or ``tent-like'' structure---where the direction field undergoes a macroscopic jump and local mean oscillations completely fail to decay. This is precisely the singular mechanism observed in the Moffatt-Kimura scenario for the 3D Euler equations \cite{Moffatt2019, Moffatt2019R}, where two colliding vortex rings build an infinite-curvature tent due to the strict absence of viscous shear stress.

\medskip

In contrast, fully developed viscous turbulence is inconsistent with this topological locking. The generation of intense viscous shear stress at the apex of such a tent would inevitably tear the fold apart, forcing viscous reconnection and spatial smoothing. Hence, the fundamentals of viscous dynamics demand that local mean oscillations must exhibit \emph{some decay} as $r \to 0$. In functional terms, viscous flows naturally reside in $\vmo$ (local space of vanishing mean oscillations), where the supremum of the mean oscillation converges to zero alongside the ball radius:
\begin{equation}
\lim_{r \to 0} \left( \sup_{x \in \R^3} \frac{1}{|B_r(x)|} \int_{B_r(x)} |\xiVec(y) - c_{B_r(x)}| \, dy \right) = 0.
\end{equation}

Since the logarithmic weight $\phi(r) = 1/|\log r|$ fails the Dini test, the geometric condition $\xiVec \in \bmo_{\phi}$ permits highly oscillatory phase defects, merely requiring that the local mean oscillations do not flatline into an inviscid scale-invariant trap. As visualized in Figure \ref{fig:comic_panels}, this log-weighted space captures the chaotic limit of viscous structures. It accommodates the slowness of mean oscillation decay demanded by intense turbulent intermittency, yet provides enough mathematical relaxation to trigger the geometric depletion of the vortex stretching.
\end{remark}

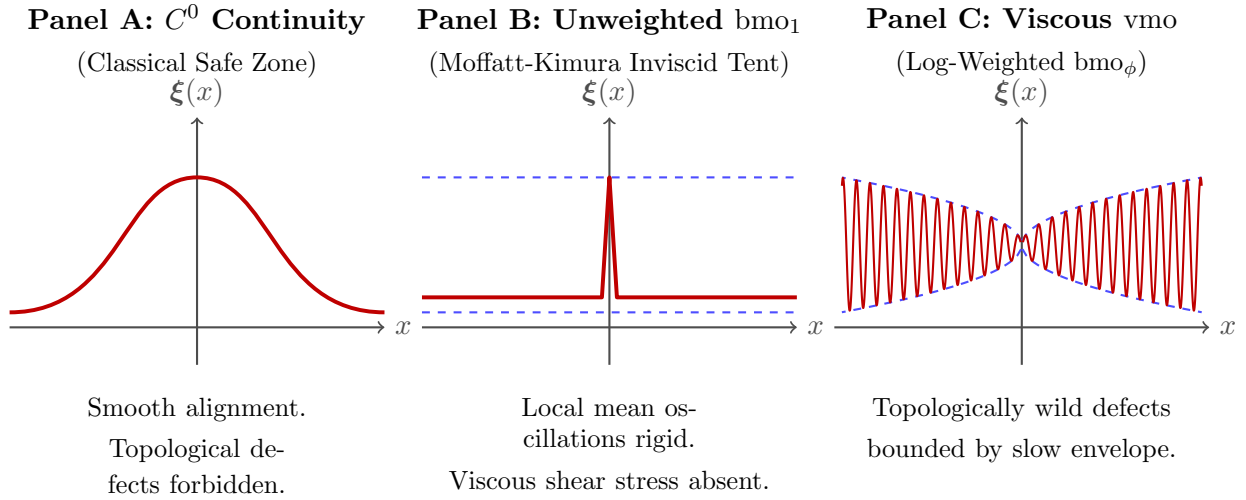
\begin{figure}[htbp]
\centering
\resizebox{\textwidth}{!}{
\begin{tikzpicture}[x=1cm, y=1cm]
  \tikzset{
    envelope/.style={dashed, thick, color=blue!70},
    signal/.style={thick, color=red!75!black, join=round},
    axis/.style={->, thick, color=black!70}
  }

  \begin{scope}[xshift=0cm]
  \node[align=center, anchor=south] at (2.5, 4.2) {\textbf{Panel A: $C^0$ Continuity}\\[0.5ex] \small (Classical Safe Zone)};
  \draw[axis] (0, 1) -- (5, 1) node[right] {$x$};
  \draw[axis] (2.5, 0.5) -- (2.5, 3.8) node[above] {$\xi(x)$};
  \draw[signal, line width=1.5pt] (0, 1.2) .. controls (1.5, 1.2) and (1.5, 3) .. (2.5, 3) .. controls (3.5, 3) and (3.5, 1.2) .. (5, 1.2);
  \node[align=center, anchor=north, text width=5.4cm] at (2.5, 0.2) {\small Smooth alignment.\\[0.5ex] Topological defects forbidden.};
  \end{scope}

  \begin{scope}[xshift=5.5cm]
  \node[align=center, anchor=south] at (2.5, 4.2) {\textbf{Panel B: Unweighted $\bmo$}\\[0.5ex] \small (Moffatt-Kimura Inviscid Tent)};
  \draw[axis] (0, 1) -- (5, 1) node[right] {$x$};
  \draw[axis] (2.5, 0.5) -- (2.5, 3.8) node[above] {$\xi(x)$};
  \draw[envelope] (0, 3) -- (5, 3);
  \draw[envelope] (0, 1.2) -- (5, 1.2);
  \draw[signal, line width=1.5pt] (0, 1.4) -- (2.4, 1.4) -- (2.5, 3) -- (2.6, 1.4) -- (5, 1.4);
  \node[align=center, anchor=north, text width=5.4cm] at (2.5, 0.2) {\small Local mean oscillations rigid.\\[0.5ex] Viscous shear stress absent.};
  \end{scope}

  \begin{scope}[xshift=11cm]
  \node[align=center, anchor=south] at (2.5, 4.2) {\textbf{Panel C: Viscous $\vmo$}\\[0.5ex] \small (Log-Weighted $\bmo_{\phi}$)};
  \draw[axis] (0, 1) -- (5, 1) node[right] {$x$};
  \draw[axis] (2.5, 0.5) -- (2.5, 3.8) node[above] {$\xi(x)$};
  \draw[envelope, domain=0.02:2.4, samples=100, smooth] plot ({2.5 - \x}, {2.1 + 0.9/(1 + ln(2.4/\x))});
  \draw[envelope, domain=0.02:2.4, samples=100, smooth] plot ({2.5 + \x}, {2.1 + 0.9/(1 + ln(2.4/\x))});
  \draw[envelope, domain=0.02:2.4, samples=100, smooth] plot ({2.5 - \x}, {2.1 - 0.9/(1 + ln(2.4/\x))});
  \draw[envelope, domain=0.02:2.4, samples=100, smooth] plot ({2.5 + \x}, {2.1 - 0.9/(1 + ln(2.4/\x))});
  \draw[signal, domain=0.02:2.4, samples=1200, smooth] plot ({2.5 - \x}, {2.1 + 0.9/(1 + ln(2.4/\x)) * sin(deg(9*ln(2.4/\x)))});
  \draw[signal, domain=0.02:2.4, samples=1200, smooth] plot ({2.5 + \x}, {2.1 + 0.9/(1 + ln(2.4/\x)) * sin(deg(9*ln(2.4/\x)))});
  \fill[black] (2.5, 2.1) circle (1.2pt);
  \node[align=center, anchor=north, text width=5.4cm] at (2.5, 0.2) {\small Topologically wild defects\\[0.5ex] bounded by slow envelope.};
  \end{scope}
\end{tikzpicture}
}
\vspace{0.2cm} 
\caption{\textbf{Topological Phase Structures.} \emph{Panel A:} Uniformly continuous constraints forbid sharp defects. \emph{Panel B:} The inviscid $\bmo_1$ trap permits rigid, non-decaying topological jumps, analogous to the infinite-curvature tent of the inviscid Moffatt-Kimura scenario. \emph{Panel C:} The viscous $\bmo_{\phi}$ assumption enforces vanishing mean oscillation ($\vmo$), yet permits infinitely oscillating, topologically severe phase defects provided they are bounded by a slowly decaying logarithmic envelope.}
\label{fig:comic_panels}
\end{figure}

\section{Geometric Framework: The Unidirectional Commutator}

The scalar evolution of the vorticity magnitude $\omega = |\om|$ is governed by the advection-diffusion inequality:
\begin{equation}\label{eq:mag_evolution}
(\partial_t + \uVec \cdot \nabla - \nu \Delta) \omega \le \alpha \omega,
\end{equation}
where the stretching eigenvalue is given by the Rayleigh quotient $\alpha = \xiVec \cdot S \xiVec$, and $S = \frac{1}{2}(\nabla \uVec + \nabla \uVec^T)$ 
is the rate-of-strain tensor \cite{Constantin1990}.
\medskip

The strain tensor $S$ is recovered from the full vorticity vector $\om = \omega \xiVec$ via a matrix of homogeneous Calder\'on-Zygmund singular integral operators $\mathcal{T}$, such that $S = \mathcal{T}(\omega \xiVec)$. We isolate the geometric structure of the stretching factor $\alpha$ by adding and subtracting the local evaluation of the direction field $\xiVec(x)$ inside the singular integral:
\begin{equation}
\begin{aligned}
S(x)\xiVec(x) &= \mathcal{T}(\omega \xiVec)(x) \xiVec(x) \\
&= \underbrace{\mathcal{T}\big(\omega (\xiVec - \xiVec(x))\big)(x) \xiVec(x)}_{\text{Commutator Part}} + \underbrace{\mathcal{T}\big(\omega \xiVec(x)\big)(x) \xiVec(x)}_{\text{Unidirectional Part}}.
\end{aligned}
\end{equation}

Taking the inner product with $\xiVec(x)$ extracts the stretching eigenvalue $\alpha(x)$. A key analytic cancellation occurs in the second term: $\xiVec(x) \cdot \big[\mathcal{T}(\omega \xiVec(x))(x) \xiVec(x)\big]$. The operator $\mathcal{T}(\omega \xiVec(x))$ computes the strain at $x$ generated by a hypothetical vorticity field $\tilde{\om}(y) = \omega(y)\xiVec(x)$ which is strictly unidirectional. Fundamentaly -- as observed in \cite{Constantin1994} -- the geometric kernel of this operator is proportional to $\hat{y} \cdot (\tilde\om \times \xiVec) = 0$.

\medskip

Due to this exact cancellation dictated by the Biot-Savart law, the purely scalar unidirectional strain evaluates to zero. The stretching eigenvalue is thus governed by the commutator of the Calder\'on-Zygmund operator $\mathcal{T}$ and pointwise multiplication by the direction field $\xiVec$:
\begin{equation}\label{eq:alpha_commutator}
\alpha(x) = \xiVec(x) \cdot [\mathcal{T}, \xiVec](\omega)(x) \cdot \xiVec(x).
\end{equation}

\section{Localized Commutator Estimates in Lorentz Spaces}

In order to execute De Giorgi energy estimates for the vorticity, we bound the restricted Lorentz norm of the stretching eigenvalue $\alpha$ over shrinking super-level sets $A_\lambda(t) = \{ x \in \R^3 : \omega(x,t) > \lambda \}$. As noted in Definition \ref{def:critical_point}, for a critical point singularity ($\omega \sim |x|^{-2}$), the support $A_\lambda(t)$ is contained entirely within a spatial ball $B_R$ centered at the origin with radius $R \le C\lambda^{-1/2}$.
\medskip

The core objective of this section is to prove that the geometric regularity of $\xiVec$ drives this localized norm to zero as $R \to 0$, effectively breaking the critical scale-invariance.

\begin{theorem}[Localized Logarithmic Commutator Smoothing]\label{thm:local_crw}
Suppose the vorticity conforms to the critical concentration profile (critical point singularity) uniformly on $(T^*-\epsilon, T^*)$ where $T^*$ is the first possible singular time, in particular $\omega \in L^\infty\bigl((T^*-\epsilon, T^*); L^{3/2, \infty}(\R^3)\bigr)$, and let the vorticity direction satisfy the uniform phase regularity $\xiVec \in L^\infty\bigl((T^*-\epsilon, T^*); \bmo_{\phi}(\R^3)\bigr)$ with $\phi(r) = 1/|\log r|$. For sufficiently small $R > 0$, the localized Lorentz norm of the stretching eigenvalue $\alpha(\cdot, t)$ on the ball $B_R$ satisfies the following asymptotic bound uniformly on $(T^*-\epsilon, T^*)$:
\begin{equation}
\norm{\alpha(\cdot, t)}_{L^{3/2, \infty}(B_R)} \le \frac{C_0}{|\log R|},
\end{equation}
where the constant $C_0 > 0$ depends exclusively on the $L^\infty_t$ suprema of $\omega$ and $\xiVec$, and is independent of time $t$.
\end{theorem}

\begin{proof}
By the unidirectional geometric cancellation derived in \eqref{eq:alpha_commutator} and the unit length of the direction field, the stretching eigenvalue is pointwise bounded by the magnitude of the commutator: $|\alpha(x)| \le |[\mathcal{T}, \xiVec](\omega)(x)|$. 

\medskip

We split the vorticity source into a near-field component $\omega_{in} = \omega \chi_{B_{2R}}$ and a far-field component $\omega_{out} = \omega \chi_{\R^3 \setminus B_{2R}}$.
\medskip

\noindent \textbf{1. The Near-Field Estimate:} 

\medskip

The restriction of the direction field to the expanded ball $B_{2R}$ satisfies $\norm{\xiVec}_{\mathrm{BMO}(B_{2R})} \le C \phi(2R)$. By the classical Jones Extension Theorem for BMO functions on uniform domains \cite{Jones1980}, we can extend this restriction to a global function $\tilde{\xiVec}$ defined on all of $\R^3$ such that its unweighted global BMO norm is controlled by the local weight: $\norm{\tilde{\xiVec}}_{\mathrm{BMO}(\R^3)} \le C \phi(2R)$.
\medskip

Because we evaluate the operator for $x \in B_R$, the near-field commutator corresponds to the action of the extended vector field: $[\mathcal{T}, \xiVec](\omega_{in})(x) = [\mathcal{T}, \tilde{\xiVec}](\omega_{in})(x)$.
\medskip

The Coifman-Rochberg-Weiss theorem \cite{CRW1976} guarantees that singular integral commutators with BMO symbols are bounded on $L^p(\R^3)$ for $1 < p < \infty$. By Hunt's Marcinkiewicz-type interpolation theorem \cite{Hunt1966}, this boundedness extends to the reflexive Lorentz spaces. Since $p = 3/2 \in (1, \infty)$, the commutator maps $L^{3/2, \infty} \to L^{3/2, \infty}$ boundedly, with the operator norm scaling linearly with the BMO norm of the multiplier:
\begin{equation}
\begin{aligned}
\norm{ [\mathcal{T}, \xiVec](\omega_{in}) }_{L^{3/2, \infty}(B_R)} &\le \norm{ [\mathcal{T}, \tilde{\xiVec}](\omega_{in}) }_{L^{3/2, \infty}(\R^3)} \\
&\le C \norm{\tilde{\xiVec}}_{\mathrm{BMO}(\R^3)} \norm{\omega_{in}}_{L^{3/2, \infty}(\R^3)} \\
&\le C \phi(2R) \norm{\omega}_{L^{3/2, \infty}} \le \frac{C'}{|\log R|}.
\end{aligned}
\end{equation}

\medskip

\noindent \textbf{2. The Far-Field Estimate:} 

\medskip

For $x \in B_R$, the singularity of the kernel is safely avoided, as $|x-y| \ge |y| - |x| > |y|/2$ for all $y \in \R^3 \setminus B_{2R}$. We evaluate the far-field commutator by adding and subtracting the spatial average $c_R = \frac{1}{|B_R|}\int_{B_R} \xiVec(z) dz$:
\begin{equation}
[\mathcal{T}, \xiVec](\omega_{out})(x) = \underbrace{(c_R - \xiVec(x)) \cdot \mathcal{T}(\omega_{out})(x)}_{I_1(x)} + \underbrace{\int_{|y|>2R} K(x-y)\omega(y)(\xiVec(y) - c_R) dy}_{I_2(x)}.
\end{equation}

\textit{Term $I_1$ (The Macroscopic Strain):} 

\medskip

To bound the far-field Biot-Savart strain $\mathcal{T}(\omega_{out})(x)$, we estimate the convolution via the Lorentz-H\"older pairing $L^{3/2, \infty} \times L^{3, 1} \to L^1$:
\begin{equation}
|\mathcal{T}(\omega_{out})(x)| \le C \int_{|y|>2R} \frac{|\omega(y)|}{|y|^3} dy \le C \norm{\omega}_{L^{3/2, \infty}} \norm{|y|^{-3}\chi_{\{|y|>2R\}}}_{L^{3, 1}}.
\end{equation}
We determine the scaling of the truncated geometric kernel $f(y) = |y|^{-3}\chi_{\{|y|>2R\}}$ via its decreasing rearrangement. A direct volume calculation shows $f^*(t) = \min\big((2R)^{-3}, C t^{-1}\big)$. Computing its $L^{3,1}$ norm:
\begin{equation}
\begin{aligned}
\norm{f}_{L^{3,1}} &= \int_0^\infty t^{1/3} f^*(t) \frac{dt}{t} \\
&= \int_0^{C R^3} t^{-2/3} R^{-3} dt + \int_{C R^3}^\infty t^{-5/3} dt \\
&= C_1 R^{-2} + C_2 R^{-2} = C R^{-2}.
\end{aligned}
\end{equation}
Consequently, the macroscopic strain evaluates pointwise to a uniformly bounded vector of magnitude $O(R^{-2})$.

\medskip

We now pair this $O(R^{-2})$ uniform far-field strain with the localized $L^{3/2, \infty}(B_R)$ norm of $(c_R - \xiVec(x))$. By the John-Nirenberg property of BMO spaces \cite{JohnNirenberg1961}, higher Lebesgue space oscillations inherit the local weight over the measure of the domain:
\begin{equation}
\norm{\xiVec - c_R}_{L^2(B_R)} \le C |B_R|^{1/2} \norm{\xiVec}_{\mathrm{BMO}(B_R)} \le C R^{3/2} \phi(R).
\end{equation}
Utilizing the continuous local embedding $L^2(B_R) \hookrightarrow L^{3/2, \infty}(B_R)$, where the spatial scaling is determined by $|B_R|^{\frac{2}{3} - \frac{1}{2}} = |B_R|^{1/6} \sim R^{1/2}$, we obtain:
\begin{equation}
\norm{\xiVec - c_R}_{L^{3/2, \infty}(B_R)} \le C R^{1/2} \norm{\xiVec - c_R}_{L^2(B_R)} \le C R^{1/2} \left[ R^{3/2} \phi(R) \right] = C R^2 \phi(R).
\end{equation}
Multiplying by the uniform $R^{-2}$ far-field strain cancels the spatial dimension, isolating the logarithmic envelope: $\norm{I_1}_{L^{3/2, \infty}(B_R)} \le C \phi(R)$.
\medskip

\textit{Term $I_2$ (The Oscillation Tail):}

\medskip

To bound the integral against the spatial oscillation, we partition the exterior domain into a logarithmic dyadic shell $\Omega_{mid} = \{ 2R < |y| \le R^{1/2} \}$ and a physical macroscopic tail $\Omega_{far} = \{ |y| > R^{1/2} \}$.
\medskip

For the macroscopic tail $\Omega_{far}$, we abandon local BMO bounds and rely on the physical $L^\infty$ constraint $|\xiVec| = 1$, which gives the trivial uniform bound $|c_R - \xiVec(y)| \le 2$. The pointwise bound on $B_R$ follows the identical Lorentz-H\"older pairing used in $I_1$:
\begin{equation}
|I_{2, far}(x)| \le \int_{|y|>R^{1/2}} \frac{C}{|y|^3} |\omega(y)| dy \le C \norm{\omega}_{L^{3/2, \infty}} \norm{|y|^{-3}\chi_{\{|y|>R^{1/2}\}}}_{L^{3, 1}} \le C (R^{1/2})^{-2} = C R^{-1}.
\end{equation}
Because this bound is uniform over $B_R$, its localized Lorentz norm simply multiplies by the volume scaling $\norm{\chi_{B_R}}_{L^{3/2, \infty}} \approx R^2$. Hence, $\norm{I_{2, far}}_{L^{3/2, \infty}(B_R)} \le C R \ll \phi(R)$, rendering the physical tail deeply subdominant for small $R$.

\medskip

For the active dyadic shell $\Omega_{mid}$, we decompose the integration into concentric dyadic annuli $A_k = B_{2^{k+1}R} \setminus B_{2^k R}$ for $k=1, \dots, N$, where the threshold index is $N = \lfloor \frac{1}{2} \log_2(1/R) \rfloor$. Let $c_k = c_{B_{2^k R}}$ denote the spatial means over the expanding balls, noting that $c_1 = c_{B_{2R}}$ and $c_0 = c_R$.

\medskip

By the triangle inequality, the spatial oscillation on $A_k$ is bounded by $|\xiVec(y) - c_R| \le |\xiVec(y) - c_{k+1}| + |c_{k+1} - c_0|$. By telescoping the consecutive ball averages, the drift from the core accumulates logarithmically:
\begin{equation}
|c_{k+1} - c_0| \le \sum_{j=1}^{k+1} |c_j - c_{j-1}| \le C \sum_{j=1}^{k+1} \frac{|B_j|}{|B_{j-1}|} \frac{1}{|B_j|} \int_{B_j} |\xiVec - c_j| dy \le C \sum_{j=1}^{k+1} \phi(2^j R).
\end{equation}

For a fixed $x \in B_R$, the kernel on $A_k$ is bounded by $C(2^k R)^{-3}$. Integrating over the annulus and applying the Lorentz-H\"older pairing $L^{3/2, \infty} \times L^{3, 1}$ gives:
\begin{equation}\label{eq:I2_annulus}
\begin{aligned}
\int_{A_k} \frac{|\omega(y)|}{|y|^3} |\xiVec(y) - c_R| dy 
&\le \frac{C}{(2^k R)^3} \norm{\omega}_{L^{3/2, \infty}} \\
&\quad \times \Big( \norm{\xiVec - c_{k+1}}_{L^{3, 1}(A_k)} + \norm{c_{k+1} - c_0}_{L^{3, 1}(A_k)} \Big).
\end{aligned}
\end{equation}
The local $L^{3,1}$ norm of a constant over the support $A_k$ scales with measure as $|A_k|^{1/3} \approx 2^k R$. Moreover, by John-Nirenberg, the oscillation of $\xiVec$ over the encompassing ball $B_{k+1}$ obeys the identical scaling. Interpolating via the standard Lebesgue space $L^4$, we have:
\begin{equation}
\begin{aligned}
\norm{\xiVec - c_{k+1}}_{L^{3, 1}(B_{k+1})} 
&\le C |B_{k+1}|^{\frac{1}{3} - \frac{1}{4}} \norm{\xiVec - c_{k+1}}_{L^4(B_{k+1})} \\
&\le C (2^k R)^{1/4} \left[ (2^k R)^{3/4} \phi(2^{k+1}R) \right] \\
&\approx (2^k R) \phi(2^{k+1} R).
\end{aligned}
\end{equation}
Thus, the parenthetical term in \eqref{eq:I2_annulus} is bounded by $C(2^k R) \sum_{j=1}^{k+1} \phi(2^j R)$. Summing the contributions from all annuli gives:
\begin{equation}
|I_{2, mid}(x)| \le C \sum_{k=1}^N \frac{1}{(2^k R)^3} (2^k R) \sum_{j=1}^{k+1} \phi(2^j R) \approx \frac{C}{R^2} \sum_{k=1}^N 4^{-k} \sum_{j=1}^{k+1} \phi(2^j R).
\end{equation}
By Fubini's theorem for non-negative series, we reverse the order of summation to isolate the exponential decay of the Biot-Savart geometric scaling:
\begin{equation}
|I_{2, mid}(x)| \lesssim \frac{1}{R^2} \sum_{j=1}^{N+1} \phi(2^j R) \sum_{k=\max(1, j-1)}^N 4^{-k} \approx \frac{1}{R^2} \sum_{j=1}^{N+1} 4^{-j} \phi(2^j R).
\end{equation}
Because we deliberately truncated the shell at $N = \lfloor \frac{1}{2}\log_2(1/R) \rfloor$, the maximum dyadic radius within the sum is bounded by $2^{N+1} R \approx R^{1/2}$. Over this domain, the weight $\phi(r) = 1/|\log r|$ is monotonically increasing, so we can pull the supremum out of the series: $\phi(2^j R) \le \phi(R^{1/2}) = \frac{1}{|\log R^{1/2}|} = 2\phi(R)$. Hence,

\begin{equation}
|I_{2, mid}(x)| \le \frac{C}{R^2} 2\phi(R) \sum_{j=1}^\infty 4^{-j} = C \frac{\phi(R)}{R^2}.
\end{equation}

Since $I_{2, mid}$ evaluates to a uniformly bounded vector over $B_R$, its restricted Lorentz norm simply multiplies this pointwise bound by the volume scaling $|B_R|^{2/3} \approx R^2$. The singular scaling completely vanishes, resulting in $\norm{I_{2, mid}(\cdot, t)}_{L^{3/2, \infty}(B_R)} \le C \phi(R)$.
\medskip

Summing the completely scaled near-field and far-field components closes the localized estimate, validating the ``Logarithmic Pump''. Crucially, all constituent bounds evaluate exclusively through absolute dimensional constants and the uniform \emph{a priori} suprema $\sup_t \norm{\omega(\cdot, t)}_{L^{3/2,\infty}}$ and $\sup_t \norm{\xiVec(\cdot, t)}_{\bmo_\phi}$. Hence, the final geometric constant $C_0$ is entirely independent of time, preventing any temporal degradation of the bound as $t \to T^*$:
\begin{equation}
\norm{\alpha(\cdot, t)}_{L^{3/2, \infty}(B_R)} \le \norm{[\mathcal{T}, \xiVec](\omega_{in})}_{L^{3/2, \infty}(B_R)} + \norm{I_1}_{L^{3/2, \infty}(B_R)} + \norm{I_2}_{L^{3/2, \infty}(B_R)} \le \frac{C_0}{|\log R|}.
\end{equation}
\end{proof}

\section{Breaking Lorentz Criticality via Interpolated Energy Estimates}

With the localized logarithmic depletion secured, we break the critical $L^{3/2, \infty}$ scaling of the vorticity magnitude via adapted energy method. Standard De Giorgi level-set recurrences can be delicate for mildly critical equations due to potential sublinear volumetric losses. To bypass this technical obstruction, we employ real interpolation in Lorentz spaces. This maps the localized logarithmic geometric depletion directly onto a global energy bound, allowing us to extract the explicit structural exponent $\gamma_1$ governing the improved decay of the distribution function.
\medskip

For a high vorticity level $\lambda > 0$, define the time-dependent super-level set $A_\lambda(t) = \{ x \in \R^3 : \omega(x,t) > \lambda \}$, its measure $U_\lambda(t) = |A_\lambda(t)|$, and the truncated vorticity $\omega_\lambda = (\omega - \lambda)_+$. Testing the scalar evolution inequality \eqref{eq:mag_evolution} against $\omega_\lambda$ and integrating over $\R^3$ yields the truncated enstrophy balance:

\begin{equation}\label{eq:energy_balance}
\frac{1}{2} \frac{d}{dt} \norm{\omega_\lambda}_{L^2}^2 + \nu \norm{\nabla \omega_\lambda}_{L^2}^2 \le \underbrace{\int_{A_\lambda(t)} \alpha \omega_\lambda^2 dx}_{\text{Nonlinear Stretching}} + \underbrace{\lambda \int_{A_\lambda(t)} \alpha \omega_\lambda dx}_{\text{Linear Source}}.
\end{equation}

\subsection{Absorption of the Nonlinear Stretching}
To handle the nonlinear term, we apply the H\"older inequality in Lorentz spaces. Pairing the stretching potential with the quadratic term evaluated in $L^{3,1}$ yields:
\begin{equation}\label{eq:holder_lorentz}
\int_{A_\lambda(t)} \alpha \omega_\lambda^2 dx \le \norm{\alpha(\cdot, t)}_{L^{3/2, \infty}(A_\lambda(t))} \norm{\omega_\lambda^2}_{L^{3, 1}} = \norm{\alpha(\cdot, t)}_{L^{3/2, \infty}(A_\lambda(t))} \norm{\omega_\lambda}_{L^{6, 2}}^2.
\end{equation}

By the standard Lorentz-Sobolev embedding $\dot{H}^1(\R^3) \hookrightarrow L^{6,2}(\R^3)$, there exists a universal constant $C_S$ such that $\norm{\omega_\lambda}_{L^{6,2}}^2 \le C_S^2 \norm{\nabla \omega_\lambda}_{L^2}^2$. Thus, the nonlinear stretching term is bounded by the viscous dissipation provided the localized restricted norm of the potential is sufficiently small.

\medskip

In a purely critical baseline devoid of the geometric gain, the localized $L^{3/2, \infty}$ norm of the potential evaluates to a non-vanishing $O(1)$ constant, locking the nonlinear stretching in a scale-invariant stalemate. However, under the critical point singularity hypothesis ($\omega \sim |x|^{-2}$), the super-level set $A_\lambda(t)$ is contained within a shrinking spatial ball $B_R$ of radius $R \le C \lambda^{-1/2}$.
\medskip

By the localized logarithmic smoothing established in Theorem \ref{thm:local_crw}, the restricted norm decays uniformly in time:
\begin{equation}
\norm{\alpha(\cdot, t)}_{L^{3/2, \infty}(A_\lambda(t))} \le \norm{\alpha(\cdot, t)}_{L^{3/2, \infty}(B_{C\lambda^{-1/2}})} \le \frac{C_0}{|\log (C\lambda^{-1/2})|} \approx \frac{2C_0}{\log \lambda}.
\end{equation}

Consequently, because $C_0$ is independent of time, there exists an absolute, time-independent truncation level $\lambda_0 > 0$ such that for all $\lambda \ge \lambda_0$, the uniform bound $C_S^2 \norm{\alpha(\cdot, t)}_{L^{3/2, \infty}(A_\lambda(t))} \le \frac{\nu}{2}$ holds globally over the entire interval $(0, T^*)$. At these elevated frequencies, 
the nonlinear stretching is absorbed into the fluid dissipation:

\begin{equation}\label{eq:linear_balance}
\frac{1}{2} \frac{d}{dt} \norm{\omega_\lambda}_{L^2}^2 + \frac{\nu}{2} \norm{\nabla \omega_\lambda}_{L^2}^2 \le \lambda \int_{A_\lambda(t)} \alpha \omega_\lambda dx.
\end{equation}

\subsection{A Priori Interpolation of the Linear Source}
To address the linear source term, we apply the Lorentz-H\"older pairing:
\begin{equation}
\lambda \int_{A_\lambda(t)} \alpha \omega_\lambda dx \le \lambda \norm{\alpha(\cdot, t)}_{L^{3/2, \infty}(A_\lambda(t))} \norm{\omega_\lambda}_{L^{3, 1}}.
\end{equation}
In order to avoid the volumetric dependencies that complicate De Giorgi iterations, we evaluate the $L^{3,1}$ norm via real interpolation directly between the \emph{a priori} critical space $L^{3/2, \infty}$ and the dissipation space $L^{6,2}$.

\medskip

By the real interpolation properties, $L^{3,1} = (L^{3/2, \infty}, L^{6,2})_{\theta, 1}$. Solving the parameter equation $\frac{1}{3} = \frac{1-\theta}{3/2} + \frac{\theta}{6}$ yields $\theta = 2/3$. Consequently, we obtain the following bound:
\begin{equation}
\norm{\omega_\lambda}_{L^{3, 1}} \le C_I \norm{\omega_\lambda}_{L^{3/2, \infty}}^{1/3} \norm{\omega_\lambda}_{L^{6, 2}}^{2/3}.
\end{equation}
Because the flow exhibits a critical point singularity, the magnitude obeys the global-in-time  \emph{a priori} bound $\norm{\omega_\lambda}_{L^{3/2, \infty}} \le \norm{\omega}_{L^{3/2, \infty}} \le M_0$. Applying the Sobolev embedding gives:
\begin{equation}
\norm{\omega_\lambda}_{L^{3, 1}} \le C_I M_0^{1/3} \left( C_S \norm{\nabla \omega_\lambda}_{L^2} \right)^{2/3}.
\end{equation}

Substituting this bound into the linear source term yields:
\begin{equation}
\lambda \int_{A_\lambda(t)} \alpha \omega_\lambda dx \le \lambda \left[ \frac{2C_0}{\log \lambda} \right] \left( C_I M_0^{1/3} C_S^{2/3} \norm{\nabla \omega_\lambda}_{L^2}^{2/3} \right).
\end{equation}

We absorb the fractional gradient power into the left-hand side of \eqref{eq:linear_balance} using Young's inequality for products:
\begin{equation}
\begin{aligned}
\lambda \int_{A_\lambda(t)} \alpha \omega_\lambda dx 
&\le \frac{\nu}{4} \norm{\nabla \omega_\lambda}_{L^2}^2 + \frac{C_1}{\nu^{1/2}} \left( \frac{\lambda}{\log \lambda} M_0^{1/3} \right)^{3/2} \\
&= \frac{\nu}{4} \norm{\nabla \omega_\lambda}_{L^2}^2 + \frac{C_1 M_0^{1/2}}{\nu^{1/2}} \frac{\lambda^{3/2}}{(\log \lambda)^{3/2}}.
\end{aligned}
\end{equation}

Substituting this back into the linear balance \eqref{eq:linear_balance} yields a dissipative differential inequality for the truncated enstrophy $E_\lambda(t) = \norm{\omega_\lambda(\cdot, t)}_{L^2}^2$:
\begin{equation}\label{eq:direct_energy}
\frac{1}{2} \frac{d}{dt} E_\lambda(t) + \frac{\nu}{4} \norm{\nabla \omega_\lambda}_{L^2}^2 \le \frac{C_1 M_0^{1/2}}{\nu^{1/2}} \frac{\lambda^{3/2}}{(\log \lambda)^{3/2}}.
\end{equation}

\subsection{Genesis of the Structural Constant $\gamma_1$ via Coercive Damping}
To invert this energy bound back into the spatial distribution function, we leverage the fundamental criticality of the profile to introduce a coercive linear damping term. 
\medskip

By the standard Sobolev embedding and H\"older's inequality applied on the spatial support $A_\lambda(t)$:
\begin{equation}
E_\lambda(t) = \norm{\omega_\lambda}_{L^2}^2 \le \norm{\omega_\lambda}_{L^6}^2 |A_\lambda(t)|^{2/3} \le C_S^2 \norm{\nabla \omega_\lambda}_{L^2}^2 U_\lambda(t)^{2/3}.
\end{equation}
Because the flow exhibits a critical point singularity ($\omega \in L^\infty_t L^{3/2, \infty}_x$), the measure trivially obeys the global \emph{a priori} envelope $U_\lambda(t) \le M_0^{3/2} \lambda^{-3/2}$ uniformly in time. Hence, $U_\lambda(t)^{2/3} \le M_0 \lambda^{-1}$. Substituting this into the Sobolev bound furnishes a $\lambda$-dependent coercive lower bound on the enstrophy dissipation:
\begin{equation}
\norm{\nabla \omega_\lambda}_{L^2}^2 \ge \frac{\lambda}{C_S^2 M_0} E_\lambda(t).
\end{equation}

Injecting this lower bound into the energy balance \eqref{eq:direct_energy} transforms the relation into a first-order linear ordinary differential inequality:
\begin{equation}
\frac{d}{dt} E_\lambda(t) + \left( \frac{\nu}{2 C_S^2 M_0} \right) \lambda E_\lambda(t) \le \frac{2C_1 M_0^{1/2}}{\nu^{1/2}} \frac{\lambda^{3/2}}{(\log \lambda)^{3/2}}.
\end{equation}

Let $\mu = \frac{\nu}{2 C_S^2 M_0}$ denote the damping coefficient and $K(\lambda)$ the algebraic source term on the right. We solve this ODE on time interval  $(T^*-\epsilon, T^*)$. Because $T^*$ is assumed to be the first possible singular time, the flow is perfectly regular and bounded at the anchor time $t_0 = T^*-\epsilon$. Consequently, there exists a finite maximum amplitude $\Lambda_0 = \norm{\omega(\cdot, t_0)}_{L^\infty}$. 
\medskip

For any truncation threshold $\lambda \ge \Lambda_0$, the initial truncated vorticity vanishes, $\omega_\lambda(\cdot, t_0) \equiv 0$, yielding a zero initial enstrophy condition $E_\lambda(t_0) = 0$. Integrating the differential inequality via Gr\"onwall's lemma from $t_0$ to any $t \in (t_0, T^*)$ yields:
\begin{equation}
E_\lambda(t) \le \int_{t_0}^t e^{-\mu \lambda (t-\tau)} K(\lambda) d\tau = \frac{K(\lambda)}{\mu \lambda} \left( 1 - e^{-\mu \lambda (t - t_0)} \right) \le \frac{K(\lambda)}{\mu \lambda}.
\end{equation}

Evaluating the steady-state ratio establishes an asymptotic bound that is valid uniformly on $(T^*-\epsilon, T^*)$:
\begin{equation}
E_\lambda(t) \le \frac{4 C_1 C_S^2 M_0^{3/2}}{\nu^{3/2}} \frac{\lambda^{1/2}}{(\log \lambda)^{3/2}}.
\end{equation}

To recover the spatial distribution function, we employ Chebyshev's inequality on the super-level set $A_{2\lambda}(t)$, where $\omega_\lambda > \lambda$:
\begin{equation}
U_{2\lambda}(t) = |A_{2\lambda}(t)| \le \frac{1}{\lambda^2} \int_{A_{2\lambda}(t)} \omega_\lambda^2 dx \le \frac{1}{\lambda^2} E_\lambda(t).
\end{equation}
Substituting the dynamic steady-state enstrophy bound yields:
\begin{equation}
U_{2\lambda}(t) \le \frac{1}{\lambda^2} \left[ C_2 \frac{\lambda^{1/2}}{(\log \lambda)^{3/2}} \right] = \frac{C_2}{\lambda^{3/2} (\log \lambda)^{3/2}}.
\end{equation}
Because all bounding constants depend exclusively on uniform-in-time parameters, setting $2\lambda \mapsto \lambda$ guarantees the spatial distribution function belongs to a sub-critical envelope mapping exactly to the logarithmic exponent $\gamma_1$, uniformly on $(T^*-\epsilon, T^*)$:
\begin{equation}\label{eq:omega_distribution_final}
\sup_{t \in (T^*-\epsilon, T^*)} |\{ x \in \R^3 : \omega(x,t) > \lambda \}| \le \frac{C_{\omega}}{\lambda^{3/2} (\log \lambda)^{\gamma_1}}, \quad \text{where} \quad \gamma_1 = \frac{3}{2}.
\end{equation}

\begin{remark}[Rearrangement Implication]
By setting the level volume $v \approx \lambda^{-3/2}(\log \lambda)^{-3/2}$ and inverting for the frequency level $\lambda$, we find the decreasing rearrangement satisfies $\omega^*(v, t) \le C v^{-2/3} \log^{-1}(e/v)$.
\end{remark}

\section{Transferring the Logarithmic Gain to the Velocity Field}

To interface with the harmonic measure maximum principle in the final analytical stage, we transfer the sub-critical structural gain from the vorticity to the fluid velocity, $\uVec$. Analytically, the velocity field is enslaved to the vorticity via the weakly singular Biot-Savart integral. Because the fundamental solution of the Laplacian in $\R^3$ dictates that the Biot-Savart kernel $K(x)$ scales as $|x|^{-2}$, this convolution represents the Riesz potential of order 1, denoted $\mathcal{I}_1$.

\medskip

In the baseline critical case devoid of the logarithmic pump ($\omega \in L^{3/2, \infty}$), the classical Hardy-Littlewood-Sobolev (HLS) inequality dictates that $\mathcal{I}_1$ maps $L^{3/2, \infty}(\R^3) \to L^{3, \infty}(\R^3)$. To extract the improved logarithmic exponent without degrading the structural constants, we rely on the continuous mapping properties of the decreasing rearrangement.

\medskip

By O'Neil's convolution lemma for rearrangements in Lorentz spaces \cite{ONeil1963}, the decreasing rearrangement of the velocity, denoted $u^*(v, t)$ (where $v$ represents the 1D spatial volume measure to distinguish from dynamic time $t$), is bounded by the sum of two integral operators applied to the decreasing rearrangement of the vorticity $\omega^*(s, t)$ and the decreasing rearrangement of the Riesz kernel. Because the volume distribution of $|x|^{-2}$ in $\R^3$ yields $K^*(s) \sim s^{-2/3}$, we obtain the following bound:

\begin{equation}\label{eq:oneil_bound}
u^*(v, t) \le C \left( v^{-2/3} \int_0^v \omega^*(s, t) \, ds + \int_v^\infty s^{-2/3} \omega^*(s, t) \, ds \right).
\end{equation}

Recall that the measure of the super-level set of the vorticity scales as $\lambda^{-3/2}(\log \lambda)^{-3/2}$ uniformly on
time interval $(T^*-\epsilon, T^*)$. Analytically inverting this algebraically locked profile guarantees that the decreasing rearrangement 
satisfies $\omega^*(s, t) \le C_0 s^{-2/3} \log^{-1}(e/s)$ for small $s$, where $C_0$ is independent of $t$.
\medskip

We execute the two integral operators in \eqref{eq:oneil_bound} on this explicit envelope independently:
\begin{enumerate}
    \item \textbf{The Hardy Core Average:} The first term averages the highly concentrated inner core of the fluid profile. Substituting the rearrangement yields:
    \begin{equation}
    v^{-2/3} \int_0^v s^{-2/3} \log^{-1}(e/s) \, ds.
    \end{equation}
    For small $s$, the logarithmic weight behaves as a slowly varying function, meaning the integral is strictly dominated by the algebraic singularity. Applying integration by parts, the logarithmic factor factors out to leading order, yielding:
    \begin{equation}
    v^{-2/3} \left[ 3 s^{1/3} \log^{-1}(e/s) \Big|_0^v - \int_0^v 3 s^{1/3} \frac{1}{s \log^2(e/s)} \, ds \right] \approx 3 v^{-1/3} \log^{-1}(e/v).
    \end{equation}
    
    \item \textbf{The Macroscopic Far-Field Tail:} To evaluate the decay of the far-field strain contributions, we partition the integration at the physical boundary $s=1$:
    \begin{equation}
    \int_v^\infty s^{-2/3} \omega^*(s, t) \, ds = \int_v^1 s^{-2/3} \omega^*(s, t) \, ds + \int_1^\infty s^{-2/3} \omega^*(s, t) \, ds.
    \end{equation}
    Because the global vorticity rearrangement bounds as $\omega^*(s, t) \le M_0 s^{-2/3}$ uniformly in time, the physical tail is safely $O(1)$:
    \begin{equation}
    \int_1^\infty s^{-2/3} \omega^*(s, t) \, ds \le M_0 \int_1^\infty s^{-4/3} \, ds = 3 M_0 = O(1).
    \end{equation}
    For the singular integration $\int_v^1$, we substitute the improved logarithmic envelope. Applying integration by parts yields:
    \begin{equation}
    \left[ -3 s^{-1/3} \log^{-1}(e/s) \right]_v^1 + \int_v^1 3 s^{-1/3} \frac{1}{s \log^2(e/s)} \, ds \approx 3 v^{-1/3} \log^{-1}(e/v) + O(1).
    \end{equation}
    Because the volume measure $v$ associated with the singular core evaluation is intrinsically small ($v \to 0$), the active inverse polynomial dominates and absorbs the uniform $O(1)$ physical tail.
\end{enumerate}

Both the local averaging operator and the macroscopic far-field tail compute to the identical asymptotic envelope. Summing these bounds proves that the decreasing rearrangement of the velocity inherits the exact Zygmund weight, uniformly on $(T^*-\epsilon, T^*)$:
\begin{equation}\label{eq:u_rearrangement}
u^*(v, t) \le C v^{-1/3} \log^{-1}(e/v),
\end{equation}
which places the velocity field into the sub-critical Lorentz-Zygmund space $L^{3, \infty}(\log L)$ with a time-independent constant $C$.

\subsection{Inverting the Geometric Profile to Physical Space}
To directly interface with the 1D geometric scale of sparseness utilized in the harmonic measure argument, we must invert the decreasing rearrangement \eqref{eq:u_rearrangement} back to the spatial distribution function.

\medskip

By definition of the decreasing rearrangement, the volume measure of the super-level set $v = |\{ x \in \R^3 : |\uVec(x,t)| > \lambda \}|$ satisfies $\lambda = u^*(v, t)$. Setting the equation:
\begin{equation}
\lambda \approx C v^{-1/3} \log^{-1}(e/v),
\end{equation}
we analytically invert for $v$. To leading order, cubing both sides gives $\lambda^3 \approx C^3 v^{-1} \log^{-3}(e/v)$. Since the logarithmic correction is slowly varying, taking the logarithm of the algebraic relation $v \sim \lambda^{-3}$ gives $\log(e/v) \approx 3 \log \lambda$. Substituting this back into the algebraic scaling isolates the volume measure $v$, establishing that the structural logarithmic parameter cubes uniformly on $(T^*-\epsilon, T^*)$:
\begin{equation}\label{u-log}
|\{ x \in \R^3 : |\uVec(x,t)| > \lambda \}| \le \frac{C_u}{\lambda^{3} (\log \lambda)^{3}},
\end{equation}
where $C_u$ is a time-independent constant.
\medskip

Physically, this analytical inversion demonstrates that the local fluid velocity satisfies $|\uVec(x,t)| \approx |x|^{-1} |\log|x||^{-1}$ near the singular core. Because the purely critical envelope behaves as $|x|^{-1}$, this verifies that the logarithmic pump starves the localized advective transport. The profile is forced into a regime sufficiently weak that the scale of geometrical sparseness it induces decays faster than the scale of uniform local diffusive analyticity, setting the final stage for the harmonic measure maximum principle.

\section{Singularity Evasion via the Harmonic Measure Maximum Principle}

In this final section, we utilize the logarithmic improvement on the velocity distribution function \eqref{u-log} to avert the potential singularity formation.
\medskip

The philosophical core of this endgame lies in a dynamical competition of spatial scales. The fluid velocity possesses a local-in-time uniform radius of spatial analyticity, $\rho_s$, generated by parabolic viscous diffusion. Simultaneously, the macroscopic distribution function controls the scale of local spatial sparseness, $r_s$, which characterizes the topological fragmentation of the advective transport. By demonstrating that the logarithmic pump forces $r_s$ to decay faster than the diffusive scale $\rho_s$ as the amplitude diverges, we place the flow into the domain of the harmonic measure maximum principle, generating a contradiction with the assumption of a finite-time blow-up.

\subsection{Analyticity, Sparseness, and the Harmonic Measure}

We begin by formally defining local geometric sparseness \cite{Grujic2013}, which mathematically captures the spatial intermittency of the fluid structure.
\begin{definition}
For a spatial point $x_0 \in \R^3$ and a density parameter $\delta \in (0,1)$, an open set $S$ is \emph{3D $\delta$-sparse around $x_0$ at scale $r$} if its relative volumetric density satisfies $\frac{|S \cap B_r(x_0)|}{|B_r(x_0)|} \le \delta$. Furthermore, $S$ is \emph{1D $\delta$-sparse around $x_0$ at scale $r$} if there exists a unit vector $\nu$ such that the linear density satisfies $\frac{|S \cap (x_0 - r\nu, x_0 + r\nu)|}{2r} \le \delta$.
\end{definition}
By standard geometric restriction, local 3D $\delta$-sparseness at a given scale implies 1D $\delta^{1/3}$-sparseness at the identical scale. The open sets of interest are the super-level sets of the velocity magnitude. For a fixed threshold parameter $\lambda \in (0,1)$ and time $t$, we define:
\begin{equation}\label{super}
 V_t = \{x \in \mathbb{R}^3: \, |\uVec(x, t)| > \lambda \|\uVec(t)\|_\infty\}.
\end{equation}

To exploit this 1D sparseness, we utilize the rotational invariance of the Navier-Stokes equations to align the direction of sparseness with a coordinate axis, immersing the 1D line into the complex plane $\C$. Since 1D projections of the velocity field are harmonic (and thus subharmonic) on the resulting complex disk, we invoke the harmonic measure maximum principle.
\begin{proposition}[Ransford \cite{Ransford1995}]\label{hm}
Let $\Omega$ be an open and connected subset of $\mathbb{C}$ such that its boundary has nonzero Hausdorff dimension and $K$ a Borel subset of the boundary. If $u$ is subharmonic on $\Omega$ satisfying $u(z) \le M_0$ on $\Omega$ and $\limsup_{z\to\zeta} u(z) \le m$ for $\zeta \in K$, then for any $z \in \Omega$:
\begin{equation}
u(z) \le m \, h(z,\Omega,K) + M_0\big(1 - h(z,\Omega,K)\big),
\end{equation}
where $h$ denotes the harmonic measure of $K$ with respect to $\Omega$, evaluated at $z$.
\end{proposition}
To quantify this measure within the unit disc $\mathbb{D}$, we utilize Solynin's extremal bound \cite{Solynin1997}: if $K \subset [-1,1]$ such that $|K| = 2\alpha$ and the origin $0 \notin K$, the harmonic measure at 0 is bounded below by the symmetric geometric configuration: 
\begin{equation}\label{sol_bound}
h(0,\mathbb{D}\setminus K,K) \ge \frac{2}{\pi} \arcsin \left( \frac{1-(1-\alpha)^2}{1+(1-\alpha)^2} \right).
\end{equation}

Finally, we quantify the competing scale of parabolic smoothing via the local-in-time uniform spatial analyticity bound.
\begin{theorem}[Gu \cite{Gu2010}]\label{radius}
Let $\uVec_0 \in L^\infty(\R^3)$ and $M \in (1, \infty)$. Then there exists a unique mild solution to the velocity formulation of the 3D Navier-Stokes system $\uVec \in C_w([0, T], L^\infty(\R^3))$ where
\[
  T \ge \frac{1}{c_1(M)} \frac{\nu}{\|\uVec_0\|^2_\infty},
\]
and for any $t \in (0, T]$ the solution $\uVec(\cdot, t)$ is the $\R^3$-restriction of a holomorphic function defined in the complex domain
\[
 \Omega_t = \biggl\{ x+iy \in \C^3: \, |y|  <  \frac{1}{c_2(M)} (\nu t)^{1/2} \biggr\}
\]
satisfying $\|\uVec(t)\|_{L^\infty(\Omega_t)}\le M \|\uVec_0\|_\infty$.
\end{theorem}
Because the local existence time scales critically as $T \sim \nu / \|\uVec_0\|_\infty^2$ and the parabolic heat kernel smoothing radius scales as $y \sim (\nu t)^{1/2}$, the uniform radius of spatial analyticity dynamically obeys dimensionally consistent scaling $\rho \sim (\nu T)^{1/2} \sim \nu / \|\uVec_0\|_\infty$.

\subsection{Parameter Synchronization and Escape Times}

To ensure complete algebraic consistency across the final calculation, we lock our parameters \emph{a priori}. We set the baseline 3D sparseness parameter to $\delta = 3/4$, which implies a 1D sparseness density of $(3/4)^{1/3}$. The geometric void complement thus possesses a linear relative measure of $\alpha = 1 - (3/4)^{1/3}$. Substituting this directly into Solynin's extremal bound \eqref{sol_bound} yields the harmonic threshold:
\begin{equation}\label{eq:h_star}
h^* = \frac{2}{\pi} \arcsin \left( \frac{1-(3/4)^{2/3}}{1+(3/4)^{2/3}} \right).
\end{equation}
We define our macroscopic leap parameter $M > 1$ as the algebraic solution to the convex combination $\frac{1}{2} h^* + (1-h^*) M = 1$, and subsequently fix the super-level set threshold at $\lambda = \frac{1}{2M}$. The identical $M$ is applied universally in Theorem \ref{radius}.

\medskip

To capture the singularity dynamics, we define a time $t$ as an \emph{escape time} if $\|\uVec(s)\|_\infty > \|\uVec(t)\|_\infty$ for all $s \in (t, T^*)$. Local well-posedness in $L^\infty$ guarantees that every sufficently large threshold has its own escape time, plus we can always select an escape time arbitrarily close to $T^*$.

\subsection{The Main Result: Averting the Blow-Up}

The stage is now set.

\begin{theorem}\label{main}
Let $\uVec_0 \in L^\infty(\R^3)$ and consider a unique, spatially analytic solution $\uVec$ on the interval $(0, T^*)$ where $T^*$ is the first possible singular time. 
Suppose that the vorticity conforms to the critical concentration profile (critical point singularity) uniformly on $(T^*-\epsilon, T^*)$ for some $0 < \epsilon < T^*$ (in particular $\omega \in L^\infty\bigl((T^*-\epsilon, T^*); L^{3/2, \infty}(\R^3)\bigr)$).
 If the vorticity direction satisfies the uniform-in-time phase regularity $\xiVec \in L^\infty\bigl((T^*-\epsilon, T^*); \bmo_{1/|\log r|}(\R^3)\bigr)$, then $T^*$ is not a singular time and the blow-up is averted.
\end{theorem}

\begin{proof}
Let $t \in (T^*-\epsilon, T^*)$ be an escape time. By the uniform analyticity bounds of Theorem \ref{radius} applied with the parameter $M$ established above, we evolve the flow forward to a time $s = t + T_t$, where $T_t$ is the maximal local analyticity time. At this subsequent evaluation time $s$, the solution $\uVec(\cdot, s)$ is analytic with a uniform spatial radius strictly bounded below by $\frac{1}{c_3} \frac{\nu}{\|\uVec(t)\|_\infty}$. Because $t$ is an escape time ($\|\uVec(s)\|_\infty > \|\uVec(t)\|_\infty$), we may conservatively bound this scale from below via the dynamic amplitude:
\begin{equation}\label{rho_dynamic}
\rho_s = \frac{1}{c_3} \frac{\nu}{\|\uVec(s)\|_\infty}.
\end{equation}

We now compute the scale of geometric sparseness of the advective transport. By substituting the peak amplitude level into the strictly sub-critical spatial distribution function derived in \eqref{u-log}, the physical volume of the super-level set $V_s$ at threshold $\beta = \lambda \|\uVec(s)\|_\infty = \frac{1}{2M} \|\uVec(s)\|_\infty$ satisfies:

\begin{equation}
|V_s| \le \frac{C}{\lambda^3 \|\uVec(s)\|_\infty^3 \log^3\bigl(e + \lambda \|\uVec(s)\|_\infty\bigr)}.
\end{equation}
To achieve 3D $\delta$-sparseness ($\delta = 3/4$) at a spatial scale $r_s$, we require $|V_s \cap B_{r_s}| / |B_{r_s}| \le 3/4$. Trivially bounding the intersection by the total measure $|V_s|$, the required scale resolves algebraically to $r_s \approx |V_s|^{1/3}$. Extracting the cube root isolates a single logarithmic gain:
\begin{equation}\label{eq:rs_bound}
r_s \le \frac{c_4}{\|\uVec(s)\|_\infty \log\bigl(e + \lambda \|\uVec(s)\|_\infty\bigr)}.
\end{equation}

It is here that the logarithmic pump breaks the scaling balance of the Navier-Stokes equations. The diffusive analyticity radius scales as $\rho_s \sim \|\uVec(s)\|_\infty^{-1}$. However, due to the logarithmic depletion of the vortex stretching, the advective sparseness scale decays logarithmically faster, asymptotically satisfying $r_s \sim \|\uVec(s)\|_\infty^{-1} \log^{-1}(\|\uVec(s)\|_\infty)$.
\medskip

Consequently, as the flow attempts to form the singularity ($\|\uVec(s)\|_\infty \to \infty$), the ratio $r_s / \rho_s \sim 1/\log(\|\uVec(s)\|_\infty)$ vanishes. By choosing the escape time $t$ sufficiently close to $T^*$, we guarantee that $r_s \le \rho_s$. T

\medskip

Let $x_0 \in \R^3$ be arbitrary. We aim to show $|\uVec(x_0, s)| \le \|\uVec(t)\|_\infty$, which contradicts $t$ serving as an escape time. By translational invariance, we set $x_0 = 0$. Assume $|\uVec(0, s)| > 0$ (otherwise the desired bound holds trivially). By the rotational invariance of the equations, we orient the coordinate system such that the direction of 1D sparseness lies along the $x_1$-axis, and we immerse this line into the complex plane to define the open disk $D_{r_s} \subset \C$. 

\medskip

To apply a scalar maximum principle, we define the 1D projection along the unit vector $\boldsymbol{e} = \uVec(0,s)/|\uVec(0,s)|$ as $F(z) = \uVec(z, s) \cdot \boldsymbol{e}$. Because $\uVec(z,s)$ is holomorphic on $D_{r_s}$, its real part $v(z) = \Re(F(z))$ is a real-valued harmonic (and thus subharmonic) function on the disk. At the origin, it satisfies $v(0) = \uVec(0, s) \cdot \boldsymbol{e} = |\uVec(0, s)|$.

\medskip

We construct the boundary control set $K$ as the compact complement of $V_s \cap (-r_s, r_s)$ within $[-r_s, r_s]$. Due to the established 1D sparseness, the Lebesgue measure of the void satisfies $|K| \ge 2r_s(1 - (3/4)^{1/3})$.

\medskip

To execute the harmonic measure estimate on the subharmonic function $v(z), $ we analyze two distinct phase states:
\begin{enumerate}
    \item \textbf{Case $0 \in K$ (The Origin is in the Void):} By definition of the complement set $K$, the physical origin $x=0$ lies outside the super-level set $V_s$. Thus, the velocity magnitude trivially satisfies $|\uVec(0, s)| \le \lambda \|\uVec(s)\|_\infty = \frac{1}{2M} \|\uVec(s)\|_\infty$. By the $L^\infty$ evolution bound of Theorem \ref{radius} on the interval $[t, s]$, we have $\|\uVec(s)\|_\infty \le M \|\uVec(t)\|_\infty$. Therefore, $|\uVec(0, s)| \le \frac{1}{2}\|\uVec(t)\|_\infty$. This contradicts the definition of $t$ as an escape time.
    
    \item \textbf{Case $0 \notin K$ (The Origin is in the Core):} Because the harmonic measure is scale-invariant with respect to the dilation $z \to z/r_s$, Solynin's extremal property \cite{Solynin1997} immediately guarantees that the harmonic measure of the complement set evaluated at the origin satisfies $h(0, D_{r_s}, K) \ge h^*$. 
    
    We apply Ransford's maximum principle \cite{Ransford1995} (Proposition \ref{hm}) to the subharmonic function $v(z)$. For any real point $x \in K$, we have $x \notin V_s$, which implies $|\uVec(x, s)| \le \frac{1}{2M}\|\uVec(s)\|_\infty$. Consequently, $v(x) \le |F(x)| \le |\uVec(x, s)| \le \frac{1}{2M}\|\uVec(s)\|_\infty$. Thus, the threshold on $K$ is $m = \frac{1}{2M}\|\uVec(s)\|_\infty$. The global bound on $D_{r_s}$ is $M\|\uVec(t)\|_\infty$ (again via Theorem \ref{radius}). We obtain:
    \begin{equation}
    \begin{aligned}
    |\uVec(0, s)| = v(0) 
    &\le h^* \left( \frac{1}{2M} \|\uVec(s)\|_\infty \right) \\
    &\quad + (1 - h^*) \Big( M \|\uVec(t)\|_\infty \Big).
    \end{aligned}
    \end{equation}
    Substituting the maximal evolution bound $\|\uVec(s)\|_\infty \le M \|\uVec(t)\|_\infty$ into the first term yields:
    \begin{equation}
    |\uVec(0, s)| \le \left[ \frac{1}{2} h^* + (1 - h^*) M \right] \|\uVec(t)\|_\infty.
    \end{equation}
    By the algebraic rigidification of our parameters, the bracketed coefficient identically evaluates to $1$. Thus, $|\uVec(0, s)| \le \|\uVec(t)\|_\infty$. Since the spatial origin was arbitrary, we conclude that $\|\uVec(s)\|_\infty \le \|\uVec(t)\|_\infty$, yielding the contradiction.
\end{enumerate}

In all topological configurations, the fluid forbids the existence of the subsequent state $s$. The assumption that $T^*$ acts as a blow-up time is thereby proven false, and the finite-time singularity is averted.
\end{proof}

\section{Acknowledgments}
The work is supported in part by the National Science Foundation grant DMS 2307657.

\end{document}